\documentclass[11pt]{article}
\usepackage{amsmath}
\usepackage{amssymb}
\usepackage{amsfonts}
\usepackage{mathrsfs}
\usepackage[usenames]{color}
\usepackage[numbers]{natbib}
\usepackage{amscd} 

\usepackage[dvipdfmx]{hyperref}
\usepackage[normalem]{ulem} 
\usepackage{comment}
\usepackage{amsthm}
\usepackage{enumerate}
\usepackage[dvipdfmx]{pict2e}
\usepackage{bbm,cases}
\usepackage{color}
\usepackage{bigints}
\usepackage{graphicx}

\usepackage{mathptmx}

\newtheorem{dfn}{Definition}[section]
\newtheorem{thm}[dfn]{Theorem}
\newtheorem{lem}[dfn]{Lemma}

\newtheorem{rem}[dfn]{Remark}
\newtheorem{prop}[dfn]{Proposition}
\newtheorem{thma}{Theorem}

\newtheorem{conja}{Conjecture}
\newtheorem{conjb}{Conjecture}

\makeatletter

             \@addtoreset{equation}{section}
\makeatother

\newcommand{\N}{\mathbb{N}}

\newcommand{\R}{\mathbb{R}}

\newcommand{\dis}{\displaystyle}
\newcommand{\ve}{\varepsilon}
\allowdisplaybreaks[4]

\newcommand{\lan}[2]{\left\langle #1,#2\right\rangle}

\begin{document}
\title{Two-sided bounds on free energy of directed polymers on strongly recurrent graphs }
\author{Naotaka Kajino\thanks{Research supported in part by JSPS KAKENHI Grant Number JP18H01123.}\and Kosei Konishi\and Makoto Nakashima\thanks{Research supported by JSPS KAKENHI Grant Numbers JP18H01123, JP18K13423.}}

\date{}
\pagestyle{myheadings}
\markboth{}{Bounds on Free energy of DPRE on strongly recurrent graphs}

\maketitle


\sloppy

\begin{abstract}
We study the directed polymers in random environment on an infinite graph $G=(V,E)$ on which the underlying
random walk satisfies sub-Gaussian heat kernel bounds with spectral dimension $d_{s}$ strictly less than two.
Our goal in this paper is to show
(i) the existence and the coincidence of the quenched and the annealed free energy $F_q(\beta)$, $F_a(\beta)$ and
(ii) that $F_a(\beta)-F_q(\beta)$ is comparable to $\beta^{\frac{4}{2-d_{s}}}$ for small inverse temperature $\beta$.
\end{abstract}

\vspace{1em}
{\bf AMS 2010 Subject Classification: Primary 82D60. Secondary 82C44.}  

\vspace{1em}{\bf Key words:}  directed polymers in random environment, phase transition, fractal graph, free energy, strongly recurrent graphs.

\vspace{1em}
For a probability space $(\Omega, {\cal F},P )$, we denote by $P[X]$ the expectation of random variable $X$ with respect to $P$.  Let $\mathbb{N}_0=\{0,1,2,\cdots\}$, $\mathbb{N}=\{1,2,3,\cdots\}$, and $\mathbb{Z}=\{0,\pm 1,\pm 2,\cdots\}$.  
Let $C_{x_1,\cdots,x_p}$ or $C(x_1,\cdots,x_p)$ be a non-random constant which depends only on the parameters $x_1,\cdots,x_p$. 



 \section{Introduction}

\subsection{The model.} 
The directed polymers in random environment (DPRE) was introduced by Henley and Huse \cite{HenHus} in the physics literature to analyze an influence of random media to the shape of polymer chains. Later on, Imbrie-Spencer and Bolthausen succeeded to treat DPRE mathematically \cite{ImbSpe, Bol}. Then, a lot of progress has been achieved  by many authors \cite{AlbZho,BorCorRem,ComYos,ComShiYos,ComShiYos2}. In particular, it is known that the phase transition of delocalization-localization of polymer chain is characterized by the  free energy\cite{ComShiYos}.  The reader may refer the survey text \cite{Com}.

In the most studied model, polymer chain and random media are represented by a path of random walk on graph $\mathbb{Z}^{d}$  and time-space i.i.d.\,random variables $\omega$ respectively, and their interaction is described in terms of Gibbs measure with Hamiltonian $H(\omega,S)$. 
In this paper, we consider the model on a general infinite  graph $G=(V,E)$ with bounded degree and replace the underlying walk by a reversible random walk $S$.

To formulate our model, we introduce  a framework of a graph $G=(V,E)$ and a random walk $S=\{S_n\}_{n=0}^\infty$ with suitable conditions as in \cite{Bar}. 

\underline{\textit{ Graph and random walk}}
Let $G=(V,E)$ be an infinite graph where $V$ is the set of  vertices with  countably infinite elements and $E$ is the set of undirected edges. Let $0\in V$ be the distinguished vertex called the origin. We assume that graph does not have multiple edges. Then, each edge is identified as a pair of vertices $\lan{x}{y}$ and we denote by $N(x)=\left\{y\in V:\text{There exists $e\in E$ s.t.~}e= \lan{y}{x}\right\}$.

 We assume that \begin{align}
\text{graph  $G$ is connected,}\label{AssGra1}
\end{align}
that is, for any $x,y\in V$, there exist an $n\in \mathbb{N}$ and  a sequence of $x_0,x_1,\cdots,x_n\in V$ such that $x_0=x$, $x_n=y$ and $\lan{x_{i-1}}{x_i}\in E$ for each $i=1,\cdots,n$. 

Also, we introduce the weight $\{\mu_{xy}\}_{x,y\in V}$ on graph $G$  such that \begin{enumerate}
\item $\mu_{xy}=\mu_{yx}$.
\item $\mu_{xy}\geq 0$ for any $x,y\in V$.
\item $\mu_{xy}>0$ if and only if $\lan{x}{y}\in E$.
\end{enumerate}

We consider a random walk $S$ on weighted graph $G$ with transition probability given by \begin{align*}
p_{x,y}=P_S(S_{n+1}=y|S_n=x)=\frac{\mu_{xy}}{\mu_x},\quad y\in N(x),
\end{align*}
where $\mu_x=\dis \sum_{y\in N(x)}\mu_{xy}$ and  we denote by $(\Omega_S,\mathcal{G}_S,P_S)$ the probability space on which  $S$ is defined. In particular, $P_S^x$ denotes the law of $S$ starting from $x\in V$, and also by $P_S=P_S^0$ for simplicity.

We remark that a random walk $S$ is reversible since \begin{align*}
\mu_xp_{x,y}=\mu_yp_{y,x} .
\end{align*}

Throughout the  paper, we assume the ``bounded weights" condition:  There exists a constant $C_\mu$ such that \begin{align}C_\mu^{-1}\leq \mu_{xy}\leq C_\mu  \tag{BW} \label{BW}
\end{align}
 for any $x,y\in V$ with $\langle x, y\rangle\in E$.


\underline{\textit{Random media}}
Let $\{\omega(n,x): n\in \mathbb{N}, x\in V\}$ be i.i.d.~$\R$-valued random variables which are independent of $S$ and defined on the probability space $(\Omega,\mathcal{F},Q)$. In this paper, we assume \begin{align}
 Q[\omega(n,x)]=0,\quad \text{and }\quad Q\left[\exp\left(\beta\omega(n,x)\right)\right]=:e^{\lambda(\beta)}<\infty \text{ for all }\beta\in\R.\label{ass:omega}
 \end{align}
 
For each $n\in\N$, $z\in V$, let $\theta_{n,x}:\Omega\to \Omega$ be a time-space shift of random media: \begin{align*}
\theta_{n,z}\circ \omega(m,y)=\omega(m+n,y+z)
\end{align*}
and we define $\theta_{n,z}\circ f(\omega):=f(\theta_{n,z}\circ \omega)$ for a measurable function $f:\Omega\to \R$.

 Then, Hamiltonian is defined by \begin{align*}
 H_n(S,\omega)=\sum_{i=1}^n \omega(i,S_i)
 \end{align*}
and the polymer measure is given by \begin{align*}
P_{n,x}^\beta(dS)=\frac{1}{Z_{n,\beta}^{x}}\exp\left(\beta H_n(S,\omega)\right)P_S^x(dS)
\end{align*}
where \begin{align*}
Z^{x}_{n,\beta}:=Z_{n}^{x}=P_S^x\left[\exp\left(\beta H_n(S,\omega)\right)\right]
\end{align*}
is the normalizing constant called a \textit{partition function}.



Then, the argument in \cite[Theorem 3.2]{ComYos} assure the existence of a critical point $\beta_1\in [0,\infty]$ such that \begin{align*}
Q(W_{\infty,\beta}^{x}>0)=\begin{cases}
1\  (\text{weak disorder})\quad &\beta\in \{0\}\cup [0,\beta_1)\\
0\ (\text{strong disorder})\quad &\beta\in (\beta_1,\infty).
\end{cases}
\end{align*}
 


One then defines the free energy of the systems.

\begin{prop}\label{propFE}
Assume that for each $x\in V$, \begin{align}
\varlimsup_{n\to \infty}\frac{1}{n}\log \tilde{V}(x,n)=0,\label{eq:subexp}\tag{subExp}
\end{align}
where $\tilde{V}(x,n)$ is the number elements in the ball centered at $x$ with radius $n$.

Then, for any $x\in V$, the following limit exists and is constant $Q$-a.s.: \begin{align*}
F_{q,x}(\beta)&=\lim_{n\to \infty}\frac{1}{n}\log Z_{n,\beta}^{x}=\lim_{n\to\infty}\frac{1}{n}Q[\log Z_{n}^{x}(\beta)]\\
F_a(\beta)&=\lim_{n\to\infty}\frac{1}{n}\log Q[Z_{n,\beta}^{x}]=\lambda(\beta).
\end{align*}
In particular, $F_{q,x}(\beta)$ is constant in $x\in V$ and we denote by $F_q(\beta)=F_{q,x}(\beta)$ $(x\in V)$.

$F_{q,x}(\beta)$ and $F_a(\beta)$ are called the quenched free energy and the annealed free energy of the systems, respectively. 
\end{prop}

\begin{rem}
Jensen's inequality implies that \begin{align*}
\frac{1}{n}Q[\log Z_{n,\beta}^{x}]\leq \frac{1}{n}\log Q[Z_{n,\beta}^{x}]=\lambda (\beta)
\end{align*}
and $F_{q}(\beta)\leq \lambda(\beta)=:F_a(\beta)$. The same argument as in \cite[Theorem 3.2]{ComYos} assures the existence of critical point $\beta_2\in [0,\infty]$ such that 
\begin{align*}
F_q(\beta)-F_a(\beta)\begin{cases}
=0\quad &\beta\in [0,\beta_2]\\
<0\ (\text{very strong disorder})\quad &\beta\in (\beta_2,\infty).
\end{cases}
\end{align*}
\end{rem}

Our main result gives an asymptotic order of $F_q(\beta)-F_a(\beta)$, where the detail of assumptions is discussed in subsection \ref{sub:HK}:
\begin{thm}\label{main1}
\begin{enumerate}
\item Suppose that the underlying random walk $S$ satisfies upper heat kernel condition \eqref{UHK} with $d_{s}<2$ and \eqref{VG}. Then, there exist constants $C_1>0$ and $\beta_0>0$  such that  for any $\beta\in [0,\beta_0)$ \begin{align*}
 F_q(\beta)-F_a(\beta)\leq -C_2\beta^{\frac{4}{2-d_{s}}}.
\end{align*}
In particular, $\beta_1=\beta_2=0$.
\item In addition, we suppose that the underlying random walk $S$ satisfies lower heat kernel condition \eqref{LHK} with $d_{s}<2$. Then, there exists a constant $C_2>0$ such that  for any $\beta\in [0,\beta_0)$ \begin{align*}
-C_2 \beta^{\frac{4}{2-d_{s}}}\leq F_q(\beta)-F_a(\beta).
\end{align*}
\end{enumerate}

\end{thm}

For DPRE in $1+1$ ($d_f=1$ and $d_w=2$), the same upper bound (and lower bound for Gaussian environment) was obtained by H.~Lacoin \cite{Lac} and the same lower bound was obtained by K.S.~Alexander and G.~Yildirim \cite{AleYil}. 
\begin{rem}
  Cosco, Seroussi, and Zeitouni study the DPRE on infinite graphs independently \cite{CosSerZei}. They discuss the critical points of phase transitions for large classes of graphs. They also obtain an upper bound of asymptotics of free energy under a similar set of conditions independently.
  \end{rem}

\subsection{Heat kernel estimate}\label{sub:HK}

In this subsection, we review some selected facts on the random walk on graphs $G=(V,E)$.  

Let $B(x,n)=\{y\in V:d(x,y)\leq n\}$ be a ball centered at $x$ with radius $n$, where $d(x,y)$ is a graph distance of $x,y\in V$, that is $d(x,y)=\inf\{n\geq 0: \text{There exists a sequence $\{x_i\}_{i=0}^n\subset V$, $x=x_0,x_1,\cdots,x_n=y$ such that $\langle x_{i-1},x_i\rangle\in E$} \}$. For $x\in V$ and a nonempty subset $A,B\subset V$, we denote the distance between $x$ and $A$ by\begin{align*}
d(x,A):=\inf\{d(x,y):y\in A\}
\end{align*}
and the distance between $A$ and $B$ by \begin{align*}
d(A,B)=\inf\{d(x,y):x\in A,y\in B\}.
\end{align*}

Here, we list some properties of graphs and random walk. Let $d_f>0$ and $d_w>1$.
\begin{itemize}
\item (Upper heat kernel estimate) There exist constants $c_1,c_2>0$ such that for $n\geq 0$ and $x,y\in V$ with $d(x,y)\leq n$
\begin{align}
p_n(x,y)\leq \frac{c_1}{(n\vee 1)^{d_f/d_w}}\exp\left(-c_2\left(\frac{d(x,y)^{d_w}}{n\vee 1}\right)^{\frac{1}{d_w-1}}\right).\label{UHK}\tag{UHK}
\end{align}
\item (Lower heat kernel estimate) There exist constants $c_3,c_4>0$ such that for $n\geq 0$ and $x,y\in V$ with $d(x,y)\leq n$\begin{align}
\hspace{-4em}	\frac{c_3}{(n\vee 1)^{d_f/d_w}}\exp\left(-c_4\left(\frac{d(x,y)^{d_w}}{n\vee 1}\right)^{\frac{1}{d_w-1}}\right)
\leq p_n(x,y)+p_{n+1}(x,y)\label{LHK}\tag{LHK}
\end{align}  
\item  (Bounded geometry) The degree is bounded, that is \begin{align}
\sup_{x\in V}|N(x)|<\infty\tag{BG}\label{BG}
\end{align} 
\item (Volume growth) There exists a non-random constant $C_V>0$ s.t.~ \begin{align}
C_V^{-1}n^{d_f}\leq V(x,n)\leq C_Vn^{d_f} \quad \text{for $x\in V$ and $n\geq 1$},\tag{VG}\label{VG} 
\end{align}
where $V(x,n)$ be the volume growth function:\begin{align*}
V(x,n)=\mu(B(x,n))=\sum_{y\in B(x,n)}\mu_y\quad \text{for }x\in V,n\geq 1.
\end{align*}
\end{itemize}
When \eqref{UHK} and \eqref{LHK} are satisfied, $d_w>1$ is called the \emph{walk dimension} of $S$
and $d_{s}:=2d_{f}/d_{w}$ is called the \emph{spectral dimension} of $S$ \cite[Definition 4.14]{Bar}.


We say the weighted graph is strongly recurrent if the random walk $S$ satisfies \eqref{UHK} and \eqref{LHK} with $d_{f}<d_{w}$, or equivalently, $d_{s}=2d_{f}/d_{w}<2$. \cite[Section 1.1]{BarCouKum}

\begin{rem}
\begin{itemize}
\item \eqref{UHK} and \eqref{LHK} imply \eqref{BW}, \eqref{BG}, \eqref{VG} \cite[Lemma 4.17]{Bar}.
\item \eqref{BW} and \eqref{VG} imply that there exists a non-random constant $C_V'>0 $ s.t.~\begin{align}
{C_V'}^{-1}n^{d_f}\leq \tilde{V}(x,n)\leq C_V'n^{d_f}\quad \text{for }x\in V, n\geq 1,\label{eq:polgro}\tag{VG'}
\end{align}
where $\tilde{V}(x,n)=\sharp B(x,n)$ is the number of elements in the ball $B(x,n)$.
\end{itemize}
\end{rem}

\begin{lem}\label{lem:exit}{\cite[Lemma 4.21]{Bar}} We assume \eqref{UHK}. Then, there exist  constants $c_5,c_6$ such that for $t \geq 1$,
\begin{align}
P_x\left(S_n\not\in B(x,t n^{\frac{1}{d_w}})\right)&\leq  c_5\exp\left(-c_6t^{\frac{d_w}{d_w-1}}\right)\label{outside}\\
P_x\left(\tau(x,2t n^{\frac{1}{d_w}})<n\right)&\leq 2c_5\exp\left(-c_6t^{\frac{d_w}{d_w-1}}\right),\label{exit}
\end{align}
where $\tau(x,r)=\inf \{n\geq 0:S_n\not\in B(x,r)\}$ for $r\geq 1$.
\end{lem}

\subsubsection{Example of strongly recurrent graphs}
Here, we will consider random walk on the  Sierpinski gasket graph as an example of strongly recurrent graphs. See \cite[Section2.9]{Bar}.

Let $O=(0,0)$, $a_0=(1,0)$, $b_0=({\frac{1}{2}},\frac{\sqrt{3}}{{2}})$. Let $G_0$ be the graph  which consists of the vertices $O$, $a_0$, and $b_0$ and the edges $\langle O,a_0\rangle$, $\langle O,b_0\rangle$, and $\langle a_0,b_0\rangle$. Then, we consider the sequence of graphs $\{G_n\}_{n=0}^\infty$ defined by 
\begin{align*}
G_{n+1}=G_n\cup (G_n+2^na_0)\cup (G_n+2^nb_0),
\end{align*}
where $A+a=\{x+a:x\in A\}$. Then, the graph $\displaystyle G=\bigcup_{n=0}^\infty G_n$ is   the Sierpinski gasket graph. 
\begin{center}
\begin{figure}[h]
\includegraphics[width=5in]{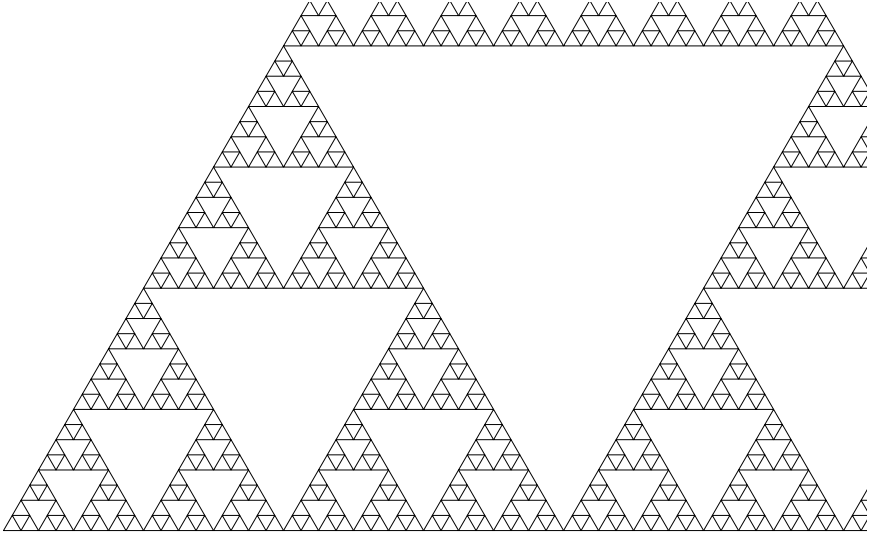}
\caption{Sierpinski gasket graph}
\end{figure} 
\end{center}
When we consider a simple random walk $S_{\textrm{SG}}$ on $G$, it is known that $S_{\textrm{SG}}$ satisfies \eqref{UHK} and \eqref{LHK} with $d_f=\frac{\log 3}{\log 2}$, $d_w=\frac{\log 5}{\log 2}$ and hence $d_{s}=\frac{\log 9}{\log 5}<1$ \cite[Corollary 6.11]{Bar}.

\subsection{Some remarks and Comments}

For DPRE in $1+1$, one of authors obtained a sharper asymptotics  as follows  \cite{Nak}. 
Suppose we assume \eqref{ass:omega} and $Q[\omega(n,x)^2]=1$ with a certain technical condition on $\omega$, then\begin{align}
\lim_{\beta\to 0}\frac{F_q(\beta)-F_a(\beta)}{\beta^4}=\lim_{T\to \infty}\frac{1}{T}\mathcal{P}[\log \mathcal{Z}_{T}]\label{eq:asymp}
\end{align} 
with $\mathcal{Z}_T=\dis \int_\mathbb{R}\mathcal{Z}(T,x)dx$, where $\mathcal{Z}(t,x)$ is a solution to  a stochastic heat equation \begin{align*}
\partial_t \mathcal{Z}=\frac{1}{2}\Delta \mathcal{Z}+\sqrt{2}\mathcal{Z}\dot{\mathcal{W}},\quad \lim_{t\to 0}\mathcal{Z}(t,x)dx=\delta (dx),
\end{align*}
where $\mathcal{W}$ is a time-space white noise on $[0,\infty)\times \mathbb{R}$.

To explain the rough idea of \eqref{eq:asymp}, we need the following theorem obtained by  T.~Alberts, K.~Khanin, and J.~Quastel\cite[Theorem 2.1]{AlbKhaQua}:
\begin{thma} 
Suppose \eqref{ass:omega} and $Q[\omega(n,x)^2]=1$. Then we have 
 \begin{align*}
\frac{Z_{Tn,\beta_n}^0}{Q\left[Z_{Tn,\beta_n}^0\right]}\Rightarrow \mathcal{Z}_T
\end{align*}
where $\beta_n=n^{-\frac{1}{4}}$.
\end{thma}
 Then, \eqref{eq:asymp} implies roughly  that\begin{align*}
(LHS)=\lim_{n\to \infty}\frac{1}{\beta_n^4}\lim_{T\to \infty}\frac{1}{Tn}Q\left[\log \frac{Z_{Tn,\beta_n}^0}{Q\left[Z_{Tn,\beta_n}^0\right]}\right]
=\lim_{T\to \infty}\lim_{n\to \infty}\frac{1}{T}Q\left[\log \frac{Z_{Tn,\beta_n}^0}{Q\left[Z_{Tn,\beta_n}^0\right]}\right]=(RHS),
\end{align*}
that is \eqref{eq:asymp} means the interchangeability of two limits in $n$  and $T$. In several discrete disordered systems, it is known that such interchangeability of limits holds \cite{BoldenH,CarGia,CarTonTor}


Thus, we have the following conjectures. 
\begin{conja}\label{conja}
Suppose the underlying random walk satisfies the local limit theorem in the sense discussed in \cite{CroHam}. Then, we have for each $T>0$\begin{align*}
\frac{Z^0_{Tn,\beta_n}}{Q\left[Z_{Tn,\beta_n}^0\right]}\Rightarrow \mathcal{Z}_T,
\end{align*} 
with $\beta_n=n^{-\frac{d_w-d_f}{2d_w}}$,  $\mathcal{Z}_T=\int_{\bar{V}}\mathcal{Z}_{T,x}\mu(dx)$, where $\mathcal{Z}_{t,x}$ satisfies the \begin{align}
\mathcal{Z}_{t,x}=p_t(x)+\hat\beta\int_0^t\int_{\bar{V}}\left( \int_{\bar{V}}p_{t-s}(x-z)\mathcal{Z}_{s,z-y}\mu(dz)\right)\mathcal{W}(ds,dy),			\label{SHEfra}
\end{align}
where $\hat{\beta}$ is a constant depends on underlying random walk and  $p_t(x)$ is a heat kernel of Brownian motion in fractal graph and $\mathcal{W}$ is a white noise on $(\bar{V},\mu)$.
\end{conja} 

\begin{rem}
Simple random walk on Sierpinski gasket graph satisfies local limit theorem\cite{CroHam}.
\end{rem}

\begin{rem}
In \cite{HinRocTep,HamYan},  SPDE on measure space equipped with gradient and divergence are discussed and they prove the existence and uniqueness of solutions. 
\end{rem}

\begin{conjb} 
Under the assumption in Conjecture \ref{conja}, we have that \begin{align*}
\lim_{\beta\to 0}\beta^{-\frac{4}{2-d_s}}(F_q(\beta)-F_a(\beta))=\lim_{T\to \infty}\frac{1}{T}\mathcal{P}\left[\log \mathcal{Z}_T\right].
\end{align*}
\end{conjb}

\section{The existence of the free energy}
In this section, we prove Proposition \ref{propFE} and   we assume \eqref{eq:subexp} and \eqref{BW}.

In \cite{ComShiYos2},  the existence of the limit $\dis \lim_{n\to\infty}\frac{1}{n}Q[\log Z_{n,\beta}^x]$ follows from superadditivity of $\{Q[\log Z_{n,\beta}^x]\}_{n=1}^\infty$, but we should remark that  $\{Q[\log Z_{n,\beta}^x]\}_{n=1}^\infty$ is not superadditive  in our case  due to loss of the shift invariance of the underlying random walk. 

Letting $Z^{x,y}_{n,\beta}$ be the point to point partition function \begin{align*}
Z^{x,y}_{n,\beta}=Z_{n}^{x,y}:=P_x\left[\exp\left(\beta \sum_{i=1}^n \omega(i,S_i)\right):S_n=y\right],
\end{align*} 
then $\{Q[\log Z^{x,x}_{n}]\}_{n=1}^\infty$ is superadditive  and hence the limit $\lim_{n\to\infty}\frac{1}{n}Q[\log Z^{x,x}_{n}]=\sup_{n\geq 1}\frac{1}{n}Q[\log Z^{x,x}_{n}]$ exists. 

\begin{proof}[Proof of Proposition \ref{propFE}]
Theorem 6.1 \cite{LiuWat}  combined with  \cite[Proposition 1.5]{ComShiYos2} gives  the concentration inequality for $\log Z_{n}^{x,y}$ and
 $\log Z_n^x$: There exists a constant $K>0$ such that for any $n\geq 1$, $t\in (-1,1)$, $\delta>0$, and $x,y\in V$
 \begin{align}
Q\left[\exp\left(t (\log Z^x_{n}-Q[\log Z^x_n])\right)\right]&\leq \exp\left(\frac{nKt^2}{1-|t|}\right)\label{concent}
\intertext{and}
Q\left(\left|\log Z_n^x-Q[\log Z_n^x]\right|>n\delta \right)&\leq \exp\left(-n\left(\sqrt{\delta+K}-\sqrt{K}\right)^2\right),\notag
\end{align}
and the same inequality holds even if $Z_n^x$ is replaced by $Z_n^{x,y}$.

Thus, it is enough to show  the convergence of $\frac{1}{n}Q[\log Z_n^x]$ and independence of $x\in V$. We will  show that \begin{align*}
\lim_{n\to\infty}\frac{1}{n}Q[\log Z_n^x]=\lim_{n\to\infty}\frac{1}{n}Q[\log Z^{x,x}_{n}]
\end{align*}
by the same argument as in \cite[Proposition 2.4]{CarHu3}: $Z_{n}^x\geq Z^x_{n,x}$ is trivial so that \begin{align*}
\lim_{n\to\infty}\frac{1}{n}Q[\log Z^{x,x}_{n}]\leq \varliminf_{n\to\infty}\frac{1}{n}Q[\log Z^x_n].
\end{align*}
Also, $Q[\log Z^{x,x}_{2n}]\geq Q[\log Z_{n}^{x,y}]+Q[\log Z_{n}^{y,x}]$ with \eqref{BW} implies that there exists a constant $c>0$ such that for any $x,y\in V$\begin{align}\label{eq:zxyzxx}
Q[\log Z^{x,x}_{2n}]\geq 2Q[\log Z_{n}^{x,y}]-c.
\end{align} 
On the other hand, we have that for any $0<t<1$ 
\begin{align}
\frac{1}{n}Q[\log Z^x_n]&\leq \frac{1}{t n}\log Q\left[\left(Z_{n}^x\right)^t\right] \hspace{10em}(\text{Jensen's inequality})\notag\\
&\leq \frac{1}{tn}\log Q\left[\sum_{y\in V}\left(Z_{n,y}^x\right)^t\right]\notag\\
&= \frac{1}{tn}\log \left(\sum_{y\in V} Q\left[\exp\left(t\left(\log Z_{n}^{x,y}-Q[\log Z_{n}^{x,y}]\right)\right)\right]\exp\left(tQ[\log Z_{n}^{x,y}]\right)\right)			\notag\\
&\leq \frac{1}{tn} \log Q\left[\sum_{y\in V}\exp\left(t \left(\log Z_{n}^{x,y}-Q[\log Z^{x,y}_{n}]\right)\right)\right]+\frac{1}{2n}Q[\log Z_{2n}^{x,x}]+\frac{c}{2nt}\notag\\
&\leq \frac{1}{tn}\left(\log \tilde{V}(x,n)+\frac{nKt^2}{1-t}\right)+\frac{1}{2n}Q[\log Z_{2n}^{x,x}]+\frac{c}{2nt},\label{plbdd}
\end{align}
where we used Jensen's inequality  in the first line, the fact $(a+b)^\theta\leq a^\theta+b^\theta$ for $a,b\geq 0$ and $\theta\in (0,1)$ in the second line, \eqref{eq:zxyzxx} in the fourth line, and  \eqref{concent}  in the last line.

Taking $n\to\infty$ and then $t\to 0$, we obtain from \eqref{eq:subexp} that \begin{align*}
\varlimsup_{n\to\infty}\frac{1}{n}Q[\log Z_n^x]\leq \lim_{n\to\infty }\frac{1}{2n}Q[\log Z_{2n}^{x,x}]
\end{align*}
and thus $\lim_{n\to\infty}\frac{1}{n}Q[\log Z_n^x]$ exists. The independence of $x$ for $F_{q,x}(\beta)$ follows by $Q[\log Z^x_{n+d(x,y)}]\geq Q[\log  Z^y_n]$.

\end{proof}

\section{Upper bound}
In this section, we prove Theorem \ref{main1} (1).
Throughout this section, we assume \eqref{BW}, \eqref{VG} (which imply \eqref{eq:polgro}), \eqref{UHK}.

We will prove the upper bound by the coarse graining scheme and change of measure method used in several polymer models \cite{AleZyg,BolCarTil,BoldenH,Lac,DerGiaLacTon,Ton}.

\subsection{Coarse graining and change of measure}

We define $W_{n,\beta}^{x}=Z_{n,\beta}^{x}/Q[Z_{n,\beta}^{x}]=P_S^x\left[e_{n}\right]$ where 
\begin{align}
e_n=e_n(\beta,\omega,S):=\exp\left(\beta H_n(S)-n\lambda(\beta)\right)
\end{align}
 and we denote by $\{\mathcal{F}_n:n\geq 0\}$ the filtration generated by $\omega(n,x)$: $\mathcal{F}_0=\{\Omega,\emptyset\}$ and $\mathcal{F}_n=\sigma[\omega(i,x): 1\leq i\leq n, x\in V]$ for $n\geq 1$.

We remark from Proposition \ref{propFE} that \begin{align}
F_q(\beta)-F_a(\beta)=\lim_{n\to\infty}\frac{1}{n}Q[\log W_n^0(\beta)]=\lim_{N\to\infty}\frac{1}{nN}Q[\log W_{nN}^0(\beta)].\label{eq:renFE}
\end{align}
Therefore, we will estimate the RHS in \eqref{eq:renFE} in the proof.

For small $\beta>0$, let $n$ be the smallest integer bigger than $h_V(\beta)$, where $h_V$ is a decreasing function on $(0,{\beta_0})$ with $\lim_{\beta\to+0}h_V(0)=\infty$ and given explicitly later.  First, we consider a covering of graph by balls with radius $n_w=n^{\frac{1}{d_w}}$. From Vitali's covering lemma, we can find  a set of verticies  $I^{(n)}=\{y_{i}^{(n)}\}_{i\in \mathbb{N}}$ such that 
\begin{enumerate}[(1)]
\item $B(y_i^{(n)},n_w)\cap B(y_j^{(n)},n_w)=\emptyset$ for $i\not=j$.
\item $\dis \bigcup_{i}B(y_i^{(n)},5n_w)=V$.
\end{enumerate}
Denoting by  $\mathcal{N}_i^{(n)}(R)$ the  elements in $I^{(n)}$ which lie in the ball $B(y_i^{(n)},Rn_w)$ for $R\geq 1$, we have \begin{align*}
\bigcup_{j\in \mathcal{N}_i^{(n)}}B(y_j^{(n)},n_w)\subset B(y_i^{(n)},(R+1)n_w)
\end{align*}
and hence \eqref{eq:polgro} implies that there exists a constant $C_C>0$ such that  \begin{align}\sharp \mathcal{N}_i^{(n)}(R)\leq C_CR^{d_f}\label{eq:ballbdd}
\end{align}  for any $n\geq 1$ and $y_{i}^{(n)}\in I^{(n)}$.


Jensen's inequality implies that for each $0<\theta<1$ \begin{align*}
\frac{1}{nN}Q[\log W_{nN}^0(\beta)]\leq \frac{1}{\theta nN}\log Q\left[W_{nN}^0(\beta)^\theta\right]	.	
\end{align*}
We will prove that there exists an $m>0$ such that \begin{align}
 Q\left[W_{nN}^0(\beta)^\theta\right]\leq 	e^{-mN}\label{eq:conc}
 \end{align} so that \begin{align*}
\varlimsup_{n\to \infty}h_V(\beta)\left(F_q(\beta)-F_a(\beta)\right)\leq -\frac{m}{\theta }.
\end{align*}

To estimate $W_{nN}^0(\beta)^\theta$, we focus on the balls which the underlying random walk passes through in each $m=n,2n,3,\cdots,nN$:  \begin{align*}
W_{nN}^0(\beta)&=P_0\left[\exp\left(\beta H_{nN}(S,\omega)-nN\lambda(\beta)\right)\right]\\
&\leq \sum_{z_1,\cdots,z_N\in I^{(n)}}P_0\left[\exp\left(\beta H_{nN}(S,\omega)-nN\lambda(\beta)\right): S_{in}\in B(z_i,5n_w), i=1,\cdots,N \right].
\end{align*}
Therefore, we have that  \begin{align}
Q\left[W_{nN}^0(\beta)^\theta\right]&\leq \sum_{z_1,\cdots,z_N\in I^{(n)}}Q\left[P_0\left[\exp\left(\beta H_{nN}(S,\omega)-nN\lambda(\beta)\right): S_{in}\in B(z_i,5n_w), i=1,\cdots,N \right]^\theta\right],\label{eq:sumt}
\end{align}
where we used the fact $(a+b)^\theta \leq a^\theta+b^\theta$ for $a,b\geq 0$ and $\theta\in (0,1)$. For simplicity of notation, we denote by \begin{align*}
W_{nN}(Z)=P_0\left[\exp\left(\beta H_{nN}(S,\omega)-nN\lambda(\beta)\right): S_{in}\in B(z_i,5n_w), i=1,\cdots,N \right].
\end{align*}

We will take $h_V(\beta)=C_1\beta^{-\frac{4}{2-d_s}}$, where $C_1>0$ will be chosen large later. Then, we have that \begin{align}
C_1\beta^{-\frac{4}{2-d_s}}\leq n<C_1\beta^{-\frac{4}{2-d_s}}+1\label{norder}
\end{align}

For each $Z=(z_1,\cdots,z_N)\in \left(I^{(n)}\right)^N$, we introduce a new probability measure which has a Radon-Nikodym derivative \begin{align*}
\frac{dQ_{Z}}{dQ}&=\exp\left(-\sum_{(l,x)\in J_Z}\left(\delta_{n}\omega(l,x)-\lambda(-\delta_{n})\right)\right)\\
&=\prod_{i=0}^{N-1}\exp\left(-\sum_{(l,x)\in J_{i,z_i}}\left(\delta_n \omega(l,x)-\lambda (-\delta_n)\right)\right),
\end{align*}
where  \begin{align*}
&J_Z=\left\{(in+k,x)\in \N\times V:i=0,\cdots,N-1,k=1,\cdots,n, x\in B(z_{i-1},C_2n_w) \right\}\\
&J_{i,z_i}=\{(in+k,x)\in \N\times V: k=1,\cdots,n, x\in B(z_{i-1},C_2n_w)\}
\end{align*}
with $z_0=0$, $C_2$ will be chosen large later, and $\delta_{n}=(C_V'n(C_2n_{w})^{d_f})^{-\frac{1}{2}}$. 

Then, H\"older's inequality yields that \begin{align}
&Q\left[P_0\left[e_{nN}: S_{in}\in B(z_i,5n_w), i=1,\cdots,N \right]^\theta\right]\notag\\
&=Q_Z\left[\frac{dQ}{dQ_Z}P_0\left[e_{nN}: S_{in}\in B(z_i,5n_w), i=1,\cdots,N \right]^\theta\right]\notag\\
&\leq Q_Z\left[\left(\frac{dQ}{dQ_Z}\right)^\frac{1}{1-\theta}\right]^{1-\theta} Q_Z\left[P_0\left[e_{nN}: S_{in}\in B(z_i,5n_w), i=1,\cdots,N \right]\right]^\theta.\label{eq:QZbdd}
\end{align}

Set $\theta=\frac{1}{2}$. Then, we know that \begin{align}
Q_Z\left[\left(\frac{dQ}{dQ_Z}\right)^2\right]&= Q\left[\exp\left(\sum_{(l,x)\in J_Z}\left(\delta_n\omega(l,x)+\lambda(-\delta_n)\right)\right)\right]\notag\\
&=\exp\left(\sum_{(l,x)\in J_Z}(\lambda(\delta_n)+\lambda(-\delta_n))\right)\notag\\
&= \exp\left(n\sum_{i=0}^{N-1}\tilde{V}(z_i,C_2n_w)(\lambda(\delta_n)+\lambda(-\delta_n))\right)\notag\\
&\leq \exp\left(NC_V'C_2^{d_f} nn_w^{d_f}(\lambda(\delta_n)+\lambda(-\delta_n))\right)\notag\\
&\leq \exp\left(NC_3C_V'C_2^{d_f} n n_w^{d_f} \delta_n^2\right)\notag\\
&\leq e^{C_3N},\label{eq:QZ}
\end{align}
where we have used $\lambda (\beta)+\lambda(-\beta)\leq C_3\beta^2 $ for small $\beta>0$ with some constant $C_3>0$. 

Markov property of $S$ implies that \begin{align}
&Q_Z\left[P_0\left[e_{nN}: S_{in}\in B(z_i,n_w), i=1,\cdots,N \right]^\frac{1}{2}\right]\notag\\
&\leq P_0\left[Q_Z\left[e_{nN}\right]
: S_{in}\in B(z_i,5n_w), i=1,\cdots,N \right]^\frac{1}{2}\notag\\
&\leq  \prod_{i=0}^{N-1}\max_{x\in B(z_i,5n_w)}P_x\left[Q\left[	\exp\left(-\sum_{l=1}^n\sum_{x\in B(z_{i},C_4n_w)}\delta_n \omega(l,x)-\lambda (-\delta_n)\right)		e_{n}	\right]	:	S_n\in B(z_{i+1},5n_w)		\right]^\frac{1}{2}\notag\\
&=\prod_{i=0}^{N-1}\max_{x\in B(z_i,5n_w)}P_x\left[\exp\left((\lambda(\beta-\delta_n)-\lambda(\beta)-\lambda(-\beta))\sharp\{1\leq l\leq n: S_l\in B(z_{i},C_2n_w)\}\right):S_n\in B(z_{i+1},5n_w)\right]^\frac{1}{2}\notag\\
&\leq \prod_{i=0}^{N-1}\max_{x\in B(z_i,5n_w)}P_x\left[\exp\left(-C_4\beta\delta_n\sharp\{1\leq l\leq n: S_l\in B(z_{i},C_2n_w)\}\right):S_n\in B(z_{i+1},5n_w)\right]^\frac{1}{2},\label{eq:ChangeMeasure}
\end{align}
where we have used $\lambda(\beta-\delta_n)-\lambda(\beta)-\lambda(-\beta)\leq -C\beta \delta_n$ for small $\beta>0$ with some constant $C_4>0$.

Putting together \eqref{eq:sumt}, \eqref{eq:QZbdd}, \eqref{eq:QZ}, and \eqref{eq:ChangeMeasure}, we get \begin{align*}
&Q\left[ W_{nN}^0(\beta)^{\frac{1}{2}}\right]\\
&\leq e^{\frac{C_3}{2}N}\left(\sup_{y\in I^{(n)}}\sum_{z\in I^{(n)}}\max_{x\in B(y,5n_w)}P_x\left[\exp\left(-C_4\beta\delta_n\sharp\{1\leq l\leq n: S_l\in B(y,C_2n_w)\}\right):S_n\in B(z,5n_w)\right]^{\frac{1}{2}}\right)^N.
\end{align*}

Thus, \eqref{eq:conc} follows when we show that for any $\ve>0$, there exists $C_1>0$, $C_2>0$ such that \begin{align*}
\sup_{y\in I^{(n)}}\sum_{z\in I^{(n)}}\max_{x\in B(y,5n_w)}P_x\left[\exp\left(-C_4\beta\delta_n\sharp\{1\leq l\leq n: S_l\in B(y,C_2n_w)\}\right):S_n\in B(z,5n_w)\right]^{\frac{1}{2}}<\ve.
\end{align*}

For $y\in I^{(n)}$ and $R\in \mathbb{N}$, we consider \begin{align*}
J_{y,R,1}^{(n)}&=\{z\in I^{(n)}: d( B(y,5n_w), B(z,5n_w))\geq Rn_w\},\\
J_{y,R,2}^{(n)}&=\{z\in I^{(n)}: d( B(y,5n_w), B(z,5n_w))< Rn_w\}.
\end{align*}
Then, we obtain from \eqref{UHK} that for $x\in B(y,5n_w)$\begin{align*}
P_x(S_n\in B(z,5n_w))&\leq \sum_{u\in B(z,5n_w)}\frac{c_1}{n^{\frac{d_f}{d_w}}}\exp\left(-c_2\left(\frac{d(x,u)^{d_w}}{n}\right)^{\frac{1}{d_w-1}}\right)\\
&\leq \tilde{V}(z,5n_w)\frac{c_1}{n^{\frac{d_f}{d_w}}}\exp\left(-c_2\left(\frac{d(B(y,5n_w),B(z,5n_w))^{d_w}}{n}\right)^{\frac{1}{d_w-1}}\right)\\
&\leq c_1C_V'5^{d_f}\exp\left(-c_2\left(\frac{d(B(y,5n_w),B(z,5n_w))^{d_w}}{n}\right)^{\frac{1}{d_w-1}}\right).
\end{align*}
Thus, we obtain that \begin{align*}
\sum_{z\in J_{y,R,1}^{(n)}}\max_{x\in B(y,5n_w)}P_x(S_n\in B(z,5n_w))^{\frac{1}{2}}&\leq \sum_{k=R}^\infty \sum_{\begin{smallmatrix}z\in I^{(n)}, \\ (k-1)n_w<d(B(y,5n_w),B(z,5n_w))\leq  kn_w\end{smallmatrix}}\max_{x\in B(y,5n_w)}P_x(S_n\in B(z,5n_w))^{\frac{1}{2}}\\
&\leq \sum_{k=R}^\infty C_C(k+10)^{d_f}\left(c_1C_V'5^{d_f}\exp\left(-c_2k^{\frac{1}{d_w-1}}\right)\right)^{\frac{1}{2}}\\
&\leq \frac{\ve}{2}
\end{align*}
by taking $R>0$ large enough since the summation in the second inequality converges.

For each $z\in J_{y,R,2}^{(n)}$, we have that 
 \begin{align*}
&P_x\left[\exp\left(-C_4\beta\delta_n\sharp\{1\leq l\leq n: S_l\in B(y,C_2n_w)\}\right):S_n\in B(z,5n_w)\right]\\
&\leq P_x\left[\exp\left(-C_4\beta\delta_n\sharp\{1\leq l\leq n: S_l\in B(y,C_2n_w)\}\right)\right] \\
&\leq P_x\left(\tau(x,C_2n_w)< n\right)+e^{-C_4\beta \delta_n n}\\
&\leq 2c_5\exp\left(-c_6\left(\frac{C_2}{2}\right)^{\frac{d_w}{d_w-1}}\right)+\exp\left(-C_4\left(\frac{C_1}{n}\right)^{\frac{d_w-d_f}{2d_w}}\left(C_V'n(C_2n_w)^{d_f}\right)^{-\frac{1}{2}}n\right)\\
&\leq 2c_5\exp\left(-c_6\left(\frac{C_2}{2}\right)^{\frac{d_w}{d_w-1}}\right)+\exp\left(		-C_4\frac{C_1^{\frac{d_w-d_f}{2d_w}}}{\sqrt{C_V'C_2^{d_f}}}\right),
\end{align*}
and hence we have that for each $y\in V$ and fixed $R>0$, 
\begin{align}
&\sum_{z\in J_{y,R,2}^{(n)}}\max_{x\in B(y,5n_w)}P_x(S_n\in B(z,5n_w))^{\frac{1}{2}}\notag\\
&\leq \sum_{z\in J_{y,R,2}^{(n)}}\left(2c_5\exp\left(-c_6\left(\frac{C_2}{2}\right)^{\frac{d_w}{d_w-1}}\right)+\exp\left(		-C_4\frac{C_1^{\frac{d_w-d_f}{2d_w}}}{\sqrt{C_V'C_2^{d_f}}}\right)\right)^{\frac{1}{2}}\notag\\
&\leq C_C(10R)^{d_f}\left(2c_5\exp\left(-c_6\left(\frac{C_2}{2}\right)^{\frac{d_w}{d_w-1}}\right)+\exp\left(		-C_4\frac{C_1^{\frac{d_w-d_f}{2d_w}}}{\sqrt{C_V'C_2^{d_f}}}\right)\right)^{\frac{1}{2}}\notag\\
&<\frac{\ve}{2}.\label{J2}
\end{align}
by taking $C_2>0$ and then $C_1>0$ large enough.


\section{Lower bound}
This section is devoted to the proof of Theorem \ref{main1} (2). Throughout this section,
we assume \eqref{UHK} and \eqref{LHK}, which imply \eqref{BW}, \eqref{BG}, \eqref{VG}, and \eqref{eq:polgro}.

Our proof of the lower bound is  a modification of the one in \cite{AleYil}. Since $G=(V,E)$ is infinite graph, there exists an inifinite path $\{x_n\}_{n=0}^\infty$ such that $x_0=0$, $\langle x_i,x_{i+1}\rangle\in  E$, and $d(0,x_n)=0$ for $n\geq 1$.

For small $\beta>0$, let $n$ be the smallest integer bigger than $h_V(\beta)=C_1\beta^{-\frac{4}{2-d_s}}$, where $C_1>0$ will be taken small later. 

We also define $n_w=n^{\frac{1}{d_w}}$ for $n\geq 1$.

We consider two sets of balls
\begin{align*}
\mathbb{B}_n&=\{B_i^{(n)}:=B(x_{in_w},{n_w}): i\geq 0\}\\
\tilde{\mathbb{B}}_n&=\{\tilde{B}_i^{(n)}:={B}(x_{in_w},{C_7n_w}): i\geq 0\},
\end{align*}
where $C_7\geq 5$ will be chosen large enough later.

Then, it is clear that \begin{alignat*}{2}
&B_i^{(n)}\cap B_{j}^{(n)}\not=\emptyset &\quad (i\not=j)\\
&\tilde{B}_{i}^{(n)}\cap \tilde{B}_{j}^{(n)}=\emptyset\quad &(|i-j|\geq 2C_7+2)\\
&B_{j}^{(n)}\subset  \tilde{B}_i^{(n)} &\quad (|i-j|\leq 1).
\end{alignat*}

For each $I, J\in\mathbb{N}$ with $0\leq J\leq I$, we set a rectangle in $\mathbb{N}_0\times V$\begin{align*}
R^{(n)}(I,J)=[In,(I+1)n]\times \tilde{B}_J^{(n)}.
\end{align*}

We introduce a coarse grained time-space lattice embedded into $\mathbb{N}_0\times \mathbb{Z}$:\begin{align*}
\mathbb{L}_{CG}&=\{(I,J)\in \mathbb{N}_0\times \mathbb{N}_0:0\leq J\leq I, I-J\in 2\mathbb{N}_0 \}.
\end{align*}

A path $\Gamma=\Gamma_{(I,J)}$ from $(0,0)$ to $(I,J)$ in $\mathbb{L}_{CG}$ is a sequence of sites  $\{(i,\gamma_i):i=0,\cdots,I\}$ with $\gamma_0=0$, $|\gamma_{i+1}-\gamma_i|=1$ ($i=0,\cdots,I-1$), and $\gamma_I=J$.  Also, an infinite path $\Gamma=\Gamma_{\infty}$ is an infinite sequence of sites  $\{(i,\gamma_i):i\geq 0\}$ with $\gamma_0=0$, $|\gamma_{i+1}-\gamma_i|=1$ ($i\geq 0$). 
A length of a path $\Gamma_{(I,J)}$ denoted by $|\Gamma|$ is $I$  and we define $|\Gamma|=\infty$ for an infinite path $\Gamma$.
We denote by  $\Gamma(i)=\gamma_i$ the spatial site of a path $\Gamma$ at time $i$.  
For $\Gamma_{(I,J)}$ and $\Gamma_{(I,J)}'$, we say that $\Gamma_{(I,J)}$ is closer to $0$ than $\Gamma_{(I,J)}'$ if $\Gamma(i)\leq \Gamma'(i)$ for $0\leq i\leq I$. Given a finite path  $\Gamma=\Gamma_{(I,J)}$, we denote by $\Gamma_+=\Gamma_{(I,J),+}$ and $\Gamma_-=\Gamma_{(I,J),-}$ the path to $(I+1,J+1)$ and $(I+1,J-1)$ whose path up to time $I$ coincides with $\Gamma$, respectively.


For an $I\geq 0$ and a path  $\Gamma$ with  $|\Gamma|\geq I$, we consider a set $\Omega_I^{(n)}(\Gamma)$ of trajectories  of random walk $S$ up to time $In$ by\begin{align*}
\Omega_I^{(n)}(\Gamma)=\left\{s=\{(i,s_i)_{i=0}^{In}\}: s_0=z,s_{Ln}\in B_{\gamma_L}^{(n)} \text{for } L\leq I, s\subset \bigcup_{L<I}R^{(n)}(L,\gamma_L) \right\},
\end{align*}
where we remark that random walk doesn't have to start at $0$ in the definition.

We obtain a lower bound of the free energy by the following lemma:
\begin{lem}\label{lem:path}
Taking $C_1>0$ small enough in $h_V(\beta)$ and $C_7>0$ large enough, then there exists a $\tilde{c}\in (0,1)$ and $p>0$ such that $\beta\in (0,\beta_0]$ \begin{align*}
Q\left(\text{There exists a random infinite path $\Gamma$ such that }W_{In}\left(\Omega_I^{(n)}(\Gamma)\right)\geq \tilde{c}^I \text{ for all $I\geq 0$}\right)>p,
\end{align*}
where we define \begin{align*}
W_n^x(A)=P_S^x\left[e_n: A\right]\quad  \text{for }A\in \mathcal{G}_S, x\in V.
\end{align*}
\end{lem}
Indeed, it is trivial that for any infinite path $\Gamma$ \begin{align*}
W_{In}\geq W_{In}\left(\Omega_I^{(n)}(\Gamma)\right)
\end{align*}
and hence Lemma \ref{lem:path} implies that \begin{align*}
Q\left(\varliminf_{I\to \infty}\frac{1}{In}\log W_{In}\geq \frac{\tilde{c}}{n}\right)>p.
\end{align*}
and Proposition \ref{propFE} tells us that $F_q(\beta)\geq \dis\frac{\tilde{c}}{n}$ $Q$-a.s.

\subsection{Proof of Lemma \ref{lem:path}}

Suppose  $1$-$0$ is assigned  to each site $(I,J)\in \mathbb{L}_{CG}$ in a certain  manner. Then, we say that $(I,J)$ is \textit{open} (\textit{closed}) if $1$ $(0)$ is assigned to $(I,J)$.  
 
We say a path $\Gamma_{(I,J)}$ is \begin{itemize}
\item \textit{open} if all the site $(i,\gamma_i)$ is open,
\item \textit{maximal} if it has the maximum number of open sites in $0\leq i\leq I-1$ among all paths to $(I,J)$,
\item \textit{optimal} if it is the maximal path which is closer to $0$ than any other paths $\Gamma_{(I,J)}'$.
\end{itemize}

For each $(I,J)\in \mathbb{L}_{CG}$, we denote by  $\Gamma_{(I,J)}^{\text{opt}}$ the  optimal path to $(I,J)$, which is uniquely determined by the configuration in the time-space site up to time $I-1$.

Now, we will assign $1$-$0$  to each site by induction in $I$. Let $\tilde{c}>0$ be a constant which will be given explicitly later. 

We assign $1$ to $(0,0)$ if \begin{align*}
W_{n}(\Omega_1^{(n)}(\Gamma_{(0,0),+}))\geq \tilde{c}
\end{align*}
 and $0$ otherwise. Given the $1$-$0$ state  to  sites $(i,j)\in \mathbb{L}_{CG}$ for $i\leq I-1$, then we assign $1$ to the site $(I,J)$ $0<J\leq I$ if \begin{align*}
\widetilde{W}_{(I,J),+}^{(n)}:= \frac{W_{(I+1)n}\left(\Omega_{I+1}^{(n)}({\Gamma_{(I,J),+}^{\text{opt}}})\right)}{W_{In}\left(\Omega_{I}^{(n)}({\Gamma_{(I,J)}^{\text{opt}}})\right)}\geq \tilde{c}\quad \text{ and \quad }
\widetilde{W}_{(I,J),-}^{(n)}:=\frac{W_{(I+1)n}\left(\Omega_{I+1}^{(n)}({\Gamma_{(I+1,J),-}^{\text{opt}}})\right)}{W_{In}\left(\Omega_I^{(n)}({\Gamma_{(I,J)}^{\text{opt}}})\right)}\geq \tilde{c},
 \end{align*}
and to the site $(I,0)$  if \begin{align*}
\widetilde{W}_{(I,0),+}^{(n)}:= \frac{W_{(I+1)n}\left(\Omega_{I+1}^{(n)}({\Gamma_{(I,0),+}^{\text{opt}}})\right)}{W_{In}\left(\Omega_I^{(n)}({\Gamma_{(I,0)}^{\text{opt}}})\right)}\geq \tilde{c}\quad,
 \end{align*}
otherwise $0$.

 The construction implies that  if the optimal path to the site $(I,J)$, $\Gamma_{(I,J)}^{{\textrm{opt}}}$, is open, then \begin{align*}
 W_{(I+1)n}\left(\Omega_I^{(n)}\left({\Gamma_{(I,J),*}^{\text{opt}}}\right)\right)\geq \tilde{c}^{I+1},\quad *\in \{+,-\}.
 \end{align*}

Also, if there exists an infinite open path $\Gamma$, then  \begin{align*}
 W_{In}\left(\Omega_I^{(n)}\left({\Gamma}\right)\right)\geq \tilde{c}^{I} \quad \text{for any $I\geq 0$}.
\end{align*}

Thus, it is enough to show that there exists $p_1>0$ such that \begin{align*}
Q\left(\text{There exists an infinite open path $\Gamma$}\right)>p_1,
\end{align*} 
where $\Gamma|_I$ is the path of $\Gamma$ up to time $I$.

We introduce  random probability measures on $B_{J}^{(n)}$ by \begin{align*}
\nu^{(n)}_{(I,J)}(x)=\frac{1}{W_{In}(\Omega_I^{(n)}(\Gamma_{(I,J)}^{\text{opt}}))}W_{In}\left((\Omega_I^{(n)}(\Gamma_{(I,J)}^{\text{opt}}))\cap \{S_{In}=x\}\right),\quad x\in B_{J}^{(n)}.
\end{align*}
Then, we find that  \begin{align*}
\widetilde{W}_{(I,J),*}^{(n)}=\sum_{x\in B_J^{(n)}}\nu_{(I,J)}^{(n)}(x) \theta_{In}\circ W^x_n\left(\{S_i\in \tilde{B}_J^{(n)},i=0,1\cdots,n\}\cap \{S_n\in B_{J*1}^{(n)}\}\right).
\end{align*}

It is easy to see that \begin{align*}
Q\left[\left.\widetilde{W}_{(I,J),*}^{(n)}\right|\mathcal{F}_{In}\right]=\sum_{x\in B_{J}^{(n)}}\nu_{(I,J)}^{(n)}(x)P_S^x\left(\{S_i\in\tilde{B}_{J}^{(n)}, i=1,\cdots,n\}  \cap \{S_n\in B_{J*1}^{(n)}\}\right)
\end{align*}
for $*\in \{+,-\}$. Let $\widetilde{\Omega}_{J*}^{(n)}(S)=\{S_i\in\tilde{B}_{J}^{(n)}, i=1,\cdots,n\}  \cap \{S_n\in B_{J*1}^{(n)}\}$.

\begin{prop}\label{prop:infprob}
There exists $c>0$ such that \begin{align*}
\inf_{n\geq 1}\inf_{J\in \mathbb{N}_0}\inf_{x\in \tilde{B_J^{(n)}}}P_S^x\left(\{S_i\in\tilde{B}_{J*}^{(n)}, i=1,\cdots,n\}  \cap \{S_n\in B_{J*1}^{(n)}\}\right)>c.
\end{align*}

\end{prop}
\begin{proof}
For fixed $n\geq 1$, $J\geq 1$, $x\in B_J^{(n)}$, \begin{align*}
P_S^x\left(\{S_i\in\tilde{B}_{J*}^{(n)}, i=1,\cdots,n\}  \cap \{S_n\in B_{J*1}^{(n)}\}\right)\geq P_S^x\left(S_n\in B_{J*1}^{(n)}\right)-P_S^x\left(S_i\not\in\tilde{B}_{J*}^{(n)}, i=1,\cdots,n\}\right).
\end{align*} 
It is easy to see that for any $J\geq 1$, $x\in B_J^{(n)}$\begin{align*}
\sum_{y\in {B}(x_{(J*1)n_w},\frac{1}{2}n_w)}\left(P_S^x\left(S_n=y\right)+P_S^x\left(S_{n+1}=y\right)\right)&\leq \sum_{y\in {B}(x_{(J*1)n_w},\frac{1}{2}n_w) }P_S^x\left(S_n=y\right)\\
&\hspace{2em}+\sum_{z\in V}\sum_{z\in {B}(x_{(J*1)n_w},\frac{1}{2}n_w+1) }P_S^x\left(S_n=z\right)P_S^z(S_1=y)\\
&\leq C_8P_S^x\left( S_n\in {B}_{J*1}^{(n)}\right),
\end{align*}
where $C_8>0$ is a constant depending only on $\mu$.
Thus, \eqref{LHK} implies that there exists a constant $C_9>0$ such that for any $J\geq 1$, $x\in B_J^{(n)}$\begin{align*}
P_S^x\left( S_n\in {B}_{J*1}^{(n)}\right)\geq C_9.
\end{align*}

From \eqref{exit},  we can take 
\begin{align*}
P_S^x\left(S_i\not\in\tilde{B}_{J*}^{(n)}, i=1,\cdots,n\}\right)\leq 2c_5\exp\left(-c_6\left(\frac{C_7-1}{2}\right)^\frac{d_w}{d_w-1}\right)
\end{align*}
small arbitrarily by letting $C_7>0$ large enough.

\end{proof}

Throughout this section, we will fix a constant $\tilde{c}=\frac{1}{2}c$, where $c>0$ is a constant appeared in Proposition \ref{prop:infprob}.

Let  $X^{(n)}_{(I,J)}=1\{(I,J)\text{ is open}\}$ for $(I,J)\in\mathbb{L}_{CG}$.  In the proof of Lemma \ref{lem:path}, we will show that the $\{X_{(I,J)}: (I,J)\in \mathbb{L}_{CG}\} $ stochastically dominates a super-critical oriented percolation.

\begin{lem}\label{lem:Berdens}
For any $\ve>0$, there exists $C_1>0$ such that for any $(I,J)\in \mathbb{L}_{CG}$ and for any small $\beta >0$ \begin{align*}
Q\left(X^{(n)}_{(I,J)}=1|\mathcal{F}_{In}\right)>1-\ve.
\end{align*}
\end{lem}

The following lemma which is a modification of \cite[Theorem 1.3]{LigSchSta} tells us that $1$-$0$-states $\{X^{(n)}_{(I,J)}:(I,J)\in \mathbb{L}_{CG}\}$ stochastically dominate non-trivial Bernoulli random variables $\{Y_{(I,J)}:(I,J)\in \mathbb{L}_{CG}\}$. 

\begin{lem}\label{lem:stodom}
Suppose that $Q\left(X^{(n)}_{(I,J)}=1\right)>p$. There exists a $p_0\in (0,1)$ such that if $p>p_0$, then there exist i.i.d.\,Bernoulli random variables $\{Y_{(I,J)}:(I,J)\in \mathbb{L}_{CG}\}$ with density $0<\rho(p)<1$ which are stochastically dominated from above by $\{X_{(I,J)}^{(n)}:(I,J)\in \mathbb{L}_{CG}\}$. 

Furthermore, $\rho(p)\to 1$ as $p\to 1$. 

\end{lem}

\begin{proof}[Proof of Lemma \ref{lem:path}]
Lemma \ref{lem:Berdens} and Lemma \ref{lem:stodom} yield that $\{X_{(I,J}^{(n)}:(I,J)\in\mathbb{L}_{CG}\}$ dominates oriented site percolation  on $\mathbb{N}_0\times \mathbb{N}_0$ with density $0<\rho(p)<1$ from above.   

In particular, we can find by the standard contour argument that  the critical probability $\overrightarrow{p}$ of oriented site percolation is non-trivial. Thus, taking $C_1>0$ small  such that $Q\left(X_{(I,J)}^{(n)}=1\right)>p$ with $\rho(p)>\overrightarrow{p}$, we have that \begin{align*}
Q\left(\textrm{There exists an infinite nearest neighbor path $\Gamma$}\right)>0.
\end{align*}

\end{proof}

\subsection{Proof of Lemma \ref{lem:Berdens}}
When we  prove that \begin{align*}
\frac{Q\left(\left.\left(\widetilde{W}_{(I,J),*}^{(n)}\right)^2\right|\mathcal{F}_{In}\right)-Q\left(\left.\widetilde{W}_{(I,J),*}^{(n)}\right|\mathcal{F}_{In}\right)^2}{Q\left(\left.\widetilde{W}_{(I,J),*}^{(n)}\right|\mathcal{F}_{In}\right)^2}<\frac{\ve}{8}\quad \textrm{ for $*\in \{+,-\}$}			
\end{align*}
$Q$-a.s., we have from  Chevyshev's inequality that \begin{align*}
Q\left(\left.\left|\widetilde{W}_{(I,J),*}^{(n)}-Q\left[\left.\widetilde{W}_{(I,J),*}^{(n)}\right|\mathcal{F}_{In}\right]\right|\geq \frac{1}{2}Q\left[\left.\widetilde{W}_{(I,J),*}^{(n)}\right|\mathcal{F}_{In}\right]\right|\mathcal{F}_{In}\right)<\frac{\ve}{2}.
\end{align*}
In particular, we have that \begin{align*}
1-\frac{\ve}{2}< Q\left(\left.\widetilde{W}_{(I,J),*}^{(n)}\geq \frac{1}{2}Q\left[\left.\widetilde{W}_{(I,J),*}^{(n)}\right|\mathcal{F}_{In}\right]\right|\mathcal{F}_{In}\right)\leq Q\left(\left.\widetilde{W}_{(I,J),*}^{(n)}>c\right|\mathcal{F}_{In}\right).
\end{align*}
Combining Proposition \ref{prop:infprob}, it is enough to show that when we take $C_1>0$ small,  \begin{align*}
Q\left(\left.\left(\widetilde{W}_{(I,J),*}^{(n)}\right)^2\right|\mathcal{F}_{In}\right)-Q\left(\left.\widetilde{W}_{(I,J),*}^{(n)}\right|\mathcal{F}_{In}\right)^2
\end{align*}
could be small.

It is easy to see from the second moment argument in DPRE that \begin{align*}
&Q\left(\left.\left(\widetilde{W}_{(I,J),*}^{(n)}\right)^2\right|\mathcal{F}_{In}\right)-Q\left(\left.\widetilde{W}_{(I,J),*}^{(n)}\right|\mathcal{F}_{In}\right)^2\\
&=\sum_{x,x'\in B_{J}^{(n)}}\nu_{(I,J)}^{(n)}(x)\nu_{(I,J)}^{(n)}(x')P_{S,S'}^{x,x'}\left[\exp\left(\gamma (\beta)L_n(S,S') \right)-1:\widetilde{\Omega}_{J*}^{(n)}(S)\cap \widetilde{\Omega}_{J*}^{(n)}(S')\right]\\
&\leq \sum_{x,x'\in B_{J}^{(n)}}\nu_{(I,J)}^{(n)}(x)\nu_{(I,J)}^{(n)}(x')P_{S,S'}^{x,x'}\left[\exp\left(\gamma (\beta)L_n(S,S') \right)-1\right]
\end{align*}
where $\gamma(\beta)=\lambda(2\beta)-2\lambda(\beta)$, $P_{S,S'}^{x,x'}$ is the product of probability measures $P_S^x$ and $P_{S'}^{x'}$, and $L_{n}(S,S')=\dis \sum_{i=1}^n1\{S_i=S_i'\}$ is the collision local time of two simple random walks $S$ and $S'$ up to time $n$.

Since $\dis \exp\left(\gamma(\beta)L_N(S,S')\right)=\prod_{i=1}^n \left(e^{\gamma(\beta) 1\{S_i=S_i'\}}-1+1\right)=\prod_{i=1}^n\left(\left(e^{\gamma(\beta)}-1\right)1\{S_i=S_i'\}+1\right)$, we have 
\begin{align*}
&P_{S,S'}^{x,x'}\left[\exp\left(\gamma (\beta)L_n(S,S') \right)-1\right]\\
&=\sum_{k=1}^n \sum_{1\leq j_1<\cdots<j_k\leq n}P_{S,S'}^{x,x'}\left[\prod_{i=1}^k \left(e^{\gamma(\beta)}-1\right)1\{S_{j_i}=S_{j_i}'\}\right]\\
&=\sum_{k=1}^n(e^{\gamma(\beta)}-1)^k\sum_{1\leq j_1<\cdots<j_k\leq n}\sum_{x_i\in V} P_{S,S'}^{x,x'}\left(S_{j_1}=S'_{j_1}=x_1\right)\prod_{i=1}^{k-1}P_S^{x_i}(S_{j_{i+1}-j_i}=x_{i+1})^2\\
&\leq \sum_{k=1}^n (e^{\gamma(\beta)}-1)^k \left(Cn^{1-\frac{d_f}{d_w}}\right) ^k,
\end{align*}
where we have used the fact that \begin{align*}
\sum_{i=1}^n\sum_{y\in V}P_{S}^x(S_i=y)^2&\leq \sum_{i=1}^n \sup_{y\in V}P_S^x(S_i=y)\sum_{z\in V}P_S^x(S_i=z)
\stackrel{\eqref{UHK}}{\leq }\sum_{i=1}^n\frac{c_1}{i^{\frac{d_f}{d_w}}}
\leq Cn^{1-\frac{d_f}{d_w}}
\end{align*}
in the last line.
Since we have that \begin{align*}
\lim_{\beta\to 0}\frac{e^{\gamma(\beta)}-1}{\beta^2}=\gamma''(0),
\end{align*}
it follows that for small  $\beta>0$\begin{align}
Q\left(\left.\left(\widetilde{W}_{(I,J),*}^{(n)}\right)^2\right|\mathcal{F}_{In}\right)-Q\left(\left.\widetilde{W}_{(I,J),*}^{(n)}\right|\mathcal{F}_{In}\right)^2&\leq \sum_{k=1}^n \left(\left(2\gamma''(0)\beta^2\right)  \left(Cn^{1-\frac{d_f}{d_w}}\right) \right)^k\notag\\
&\leq \sum_{k=1}^n \left((2\gamma''(0)\beta^2) CC_1^\frac{d_w-d_f}{d_w}\beta^{-2}\right)^{k}\notag\\
&\leq \frac{2\gamma''(0) CC_1^\frac{d_w-d_f}{d_w}}{1-2\gamma''(0) CC_1^\frac{d_w-d_f}{d_w}},\label{eq:bddlower}
\end{align}
when $C_1>0$ is taken small enough  such that $2\gamma''(0) CC_1^\frac{d_w-d_f}{d_w}<1$. Furthermore, letting $C_1>0$ small, the RHS in \eqref{eq:bddlower} could be smaller than $\frac{\ve}{8}$ and hence Lemma \ref{lem:Berdens} follows.


\appendix
\section{Proof of Lemma \ref{lem:stodom}}
First, we will introduce some notations. Let $S$ be a countable set and $\Omega^{(S)}=\{0,1\}^S$ with product topology and corresponding Borel $\sigma$-algebra, $\mathcal{B}$. We define a partial order to $\Omega^{(S)}$ by saying that for $\omega_1,\omega_2\in \Omega^{(S)}$, $\omega_1\leq \omega_2$ if
 \begin{align*}
\omega_1(s)\leq \omega_2(s),\quad \text{for all }s\in S.
\end{align*}

We say  a measurable function $f:\Omega^{(S)}\to \mathbb{R}$ is increasing if $f(\omega_1)\leq f(\omega_2)$ for any $\omega_1,\omega_2\in \Omega^{(S)}$ with $\omega_1\leq \omega_2$. 

Given two probability measures $\mu,\nu$ on $(\Omega^{(S)},\mathcal{B})$, we say that $\mu$ stochastically dominates $\nu$ ($\mu\succcurlyeq \nu$) if for any continuous increasing function $f$, \begin{align*}
\int_{\Omega^{(S)}}f(\omega)\mu(d\omega)\geq \int_{\Omega^{(S)}} f(\omega)\nu(d\omega).
\end{align*} 

For $\rho\in [0,1]$, we denote by $\pi_\rho=\dis \prod_{s\in S}\mu^{(s)}_\rho$ a probability measure on $(\Omega^{(S)},\mathcal{B})$ with  marginal distribution $\pi_\rho(\{\omega:\omega(s)=1\})=\mu^{(s)}_\rho(\omega=1)=\rho$ for $s\in S$.

We shall come back to the proof of Lemma \ref{lem:stodom}.

The followings is the first key lemma.
\begin{lem}\label{lem:stodomx}{\cite[Lemma 1]{Rus}, \cite[Lemma 1.1]{LigSchSta}} Suppose that $\{X_s:s\in S\}$ is a family of $\{0,1\}$-valued random variables, indexed by $S$, with joint law $\mu$. Suppose $S$ can be totally ordered  in such a way that, given any finite subset of $S$, $s_1<s_2<\cdots<s_{j+1}$, and any choice  of $\ve_1,\cdots,\ve_j\in \{0,1\}$, then, whenever $\mathbb{P}(X_{s_1}=\ve_1,\cdots,X_{s_j}=\ve_j)>0$,
\begin{align*}
\mathbb{P}\left(\left.X_{s_{j+1}}=1\right|X_{s_1}=\ve_1,\cdots,X_{s_j}=\ve_j\right)\geq \rho.
\end{align*}
Then, $\mu\succcurlyeq \pi_\rho$.
\end{lem}

We define a total order to $\mathbb{L}_{CG}$ by saying that \begin{align*}
(I_1,J_1)< (I_2,J_2)\quad \text{ if }\begin{cases}
I_1<I_2, \text{ or }\\
I_1=I_2 \text{ and }J_1<J_2.
\end{cases}
\end{align*}

We remark that $X_{(I,J_1)}^{(n)}$ and $X_{(I,J_2)}^{(n)}$ are  independent conditioned on $\mathcal{F}_{In}$ if $|J_1-J_2|\geq 4$. For each $(I,J)\in \mathbb{L}_{CG}$, we say that $(I',J')$ is adjacent to $(I,J)$ if $I'=I$ and $|J-J'|<4$.

The following lemma is a modification of \cite[Lemma 1.2]{LigSchSta}.
\begin{lem}\label{lem:IJdom}
We denote by $\mu^{(n)}$ the law of $\{X_{I,J}^{(n)}: (I,J)\in \mathbb{L}_{CG}\}$. Let $\ve>0$ small enough such that there exist $\alpha,r\in (0,1)$ with \begin{align*}
(1-\alpha)(1-r)^{5}\geq \ve\\
(1-\alpha)\alpha^{5}\geq \ve
\end{align*} 
Let $\{Y_{(I,J)}: (I,J)\in \mathbb{L}_{CG}\}$ be a family independent of $\{X_{I,J}^{(n)}: (I,J)\in \mathbb{L}_{CG}\}$ with joint law $\pi_r$, and let $Z_{(I,J)^{(n)}}=X_{(I,J)}^{(n)}Y_{(I,J)}$. Then, for each $(I,J)$, and for any choice of $\ve_{(i,j)}\in \{0,1\}$ with $(i,j)<(I,J)$, we have \begin{align*}
\mathbb{P}\left(Z_{(I,J)}^{(n)}=1\left|	Z_{(i,j)}^{(n)}=\ve_{(i,j)} \text{ for all }(i,j)<(I,J)	\right.\right)\geq \alpha r.
\end{align*}
\end{lem}

\begin{proof}We omit the proof since it is the same as the one in \cite{LigSchSta}.
\end{proof}
\begin{rem}
We don't know whether $\mathbb{Q}\left(X_{(I,J}^{(n)}=1| \sigma[X_{(I',J')}^{(n)}:I'>I]\right)>1-\ve$. So, we can't apply \cite[Theorem 1.3]{LigSchSta} directly to prove Proposition \ref{prop:infprob}. 
\end{rem}

\begin{proof}[Proof of Lemma \ref{lem:stodom}]
We denote by $\nu^{(n)}_{\alpha,r}$ the law of $\{Z_{(I,J)}^{(n)}: (I,J)\in \mathbb{L}_{CG}\}$. Then, we find by coupling of $X$ and $Z$ that $\mu^{(n)}\succcurlyeq \nu^{(n)}_{\alpha,r}$. Combining with Lemma \ref{lem:stodomx} and Lemma \ref{lem:IJdom} yields $\nu_{\alpha,r}^{(n)}\succcurlyeq \pi_{\alpha r}$. Letting $\ve\to 0$, we can take $\alpha r$ close to $1$ arbitrarily. 
\end{proof}



\vspace{2em}
{Naotaka Kajino, nkajino@math.kobe-u.ac.jp, Department of Mathematics, Graduate School of Science, Kobe University, Rokkodai-cho 1-1, Nada-ku, Kobe 657-8501, Japan}

{Kosei Konishi, k.kousei1206@icloud.com, Nippon Life Insurance Company, Imabashi 3-5-12, Chuo-ku, Osaka 541-8501, Japan}

{Makoto Nakashima, nakamako@math.nagoya-u.ac.jp, Graduate School of Mathematics, Nagoya University, Furocho, Chikusa-ku, Nagoya 464-8602,  Japan }

\end{document}